\documentclass[a4paper,oneside,english,reqno]{amsart}
\usepackage[utf8]{inputenc}
\usepackage[T1]{fontenc}

\DeclareFontFamily{OMX}{mlmex}{}
\DeclareFontShape{OMX}{mlmex}{m}{n}{%
   <->mlmex10%
   }{}%
\usepackage{mlmodern}

\usepackage{graphicx}
\usepackage[hscale=0.6, vscale=0.7]{geometry}
\usepackage{babel}
\usepackage[numbers]{natbib}
\usepackage{hyperref}
\hypersetup{%
  colorlinks=true,%
  citecolor=[RGB]{120,29,126},%
  pdfauthor={Jean-François Burnol},%
  pdfsubject={Moments, missing digits},%
  pdfstartview=FitH,%
  pdfpagemode=UseNone,%
}

\usepackage{amsmath,amsthm,amssymb}

\theoremstyle{plain}
\newtheorem{thm}{Theorem}[section]

\newtheorem{lemma}[thm]{Lemma}
\newtheorem{prop}[thm]{Proposition}

\theoremstyle{definition}

\usepackage{mathtools}
\DeclarePairedDelimiterX\Ifo[2]{\lbrack}{\rparen}{#1,#2}

\usepackage[np,autolanguage]{numprint}
\AtBeginDocument{\npthousandsep{\allowbreak\,}}

\newcommand\cA{\mathcal{A}}
\newcommand\cB{\mathcal{B}}

\newcommand\e{\mathsf{e}}
\newcommand\dt{\mathrm{d}t}
\newcommand\du{\mathrm{d}u}
\newcommand\dx{\mathrm{d}x}
\newcommand\dy{\mathrm{d}y}

\newcommand\ddx{\frac{\mathrm{d}}{\mathrm{d}x}}
\makeatletter
\newcommand\mapsfrom{\mathrel{\reflectbox{\m@th$\mapsto$}}}
\makeatother

\title[Oscillations of moments]{%
  The asymptotic oscillations of moments related to Dirichlet series with missing digits}

\author[J.-F. Burnol]{Jean-François Burnol}

\date{28 September 2026 (v2). Fixes some typos of v1 (28 April 2026).}
\subjclass[2020]{11N37, 44A60, 11A63 (Primary) 11M06, 11M41 (Secondary)}
\keywords{Moments, asymptotics, Hardy series, missing digits}

\allowdisplaybreaks

\begin{document}

\begin{abstract}
  We prove that the (suitably rescaled) moments of certain discrete measures
  on the unit interval, which are related to the numerical evaluation of zeta
  series with missing digits in radix $b$, are asymptotically $1$-periodic in
  the base $b$ logarithm of the index, i.e.\@ asymptotically invariant under
  multiplication by $b$ of the index.
\end{abstract}

\maketitle

\section{Introduction}

The author introduced in \cite{burnolkempner} a method of numerical evaluation
of series such as $\sum' n^{-1}$ where the symbol $\sum'$ means that the sum
is restricted to denominators not using the digit $9$ in radix $10$ (Kempner
series). This is based upon the computation of certain moments
$u_m=\int_{\Ifo{0}{1}}x^m\mu(\dx)$ of some discrete measures on the unit
interval.  For example in this ``no-$9$'' radix $10$ case, the moments are
originally defined by the formula:
\begin{equation}\label{eq:intro1}
  u_m = 0^m + \sum_{l=1}^\infty 
   \sum_{(a_1,a_2,\dots,a_l)\in \{0,\dots,8\}^l} \Bigl(\frac{a_1}{10}+\dots +\frac{a_l}{10^l}\Bigr)^m \frac1{10^l}\;.
\end{equation}
Direct numerical computation from Equation \eqref{eq:intro1} is not really feasible
for large $m$'s due to the exponential increase of the number of contributions
needed to reach a certain target precision.  We can compute high precision
values though, using a certain recurrence formula obeyed by the sequence
$(u_m)_{m\geq0}$. The recurrence involves binomial coefficients and the power sums of
the digits from $1$ to $8$. See Equation \eqref{eq:umrec},  which handles
the general quantities we will study in this paper, which involve also a
complex parameter $s$. In this introduction, $s=1$.

Figure \ref{fig:no9nolog}
plots the rescaled moments $(\frac{9}{8})^m(m+1)u_m$ for $1\leq m\leq \np{1000}$
(the value of $u_0$ is $10$).
\begin{figure}[htbp]
   \caption{$b=10$, ``no $9$'', $(\frac{9}{8})^m(m+1)u_m$ for $1\leq m \leq \np{1000}$}
\label{fig:no9nolog}
  \centering
\includegraphics[width=\linewidth]{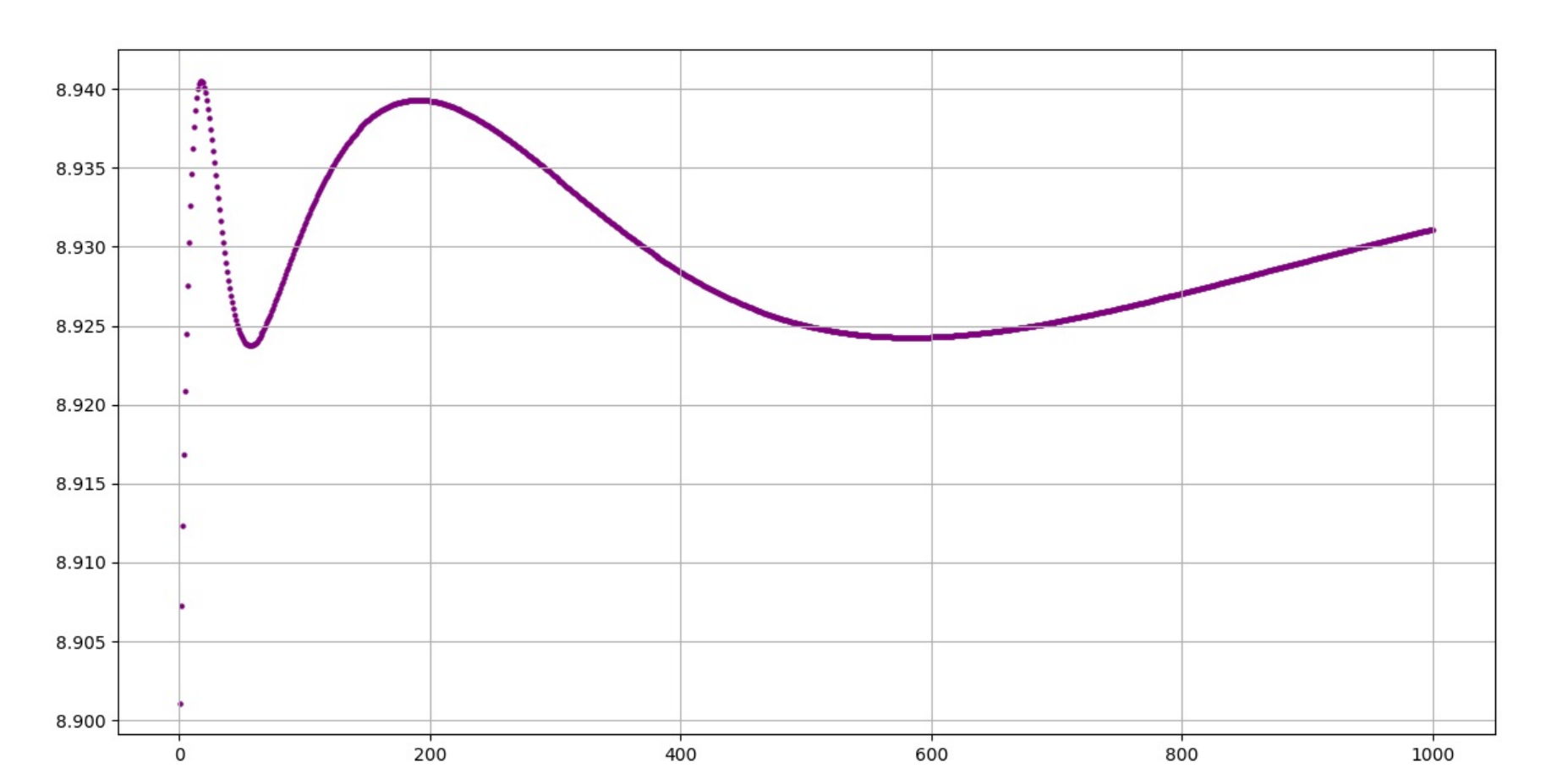}
\end{figure}
We see that, in accordance with general results of \cite{burnolkempner}, the
rescaled moments $(\frac{9}{8})^m(m+1)u_m$ appear to be bounded both above and
below.  Additionally, they seem to obey some oscillations which we can better
visualize if we use on the horizontal axis not $m$ but $\log_{10}(m)$.  See
Figure \ref{fig:no9}.
\begin{figure}[htbp]
   \caption{$b=10$, ``no $9$'', $(\frac{9}{8})^m(m+1)u_m$ vs $\log_{10}(m)$ for $1\leq m \leq \np{1000}$}
\label{fig:no9}
  \centering
\includegraphics[width=\linewidth]{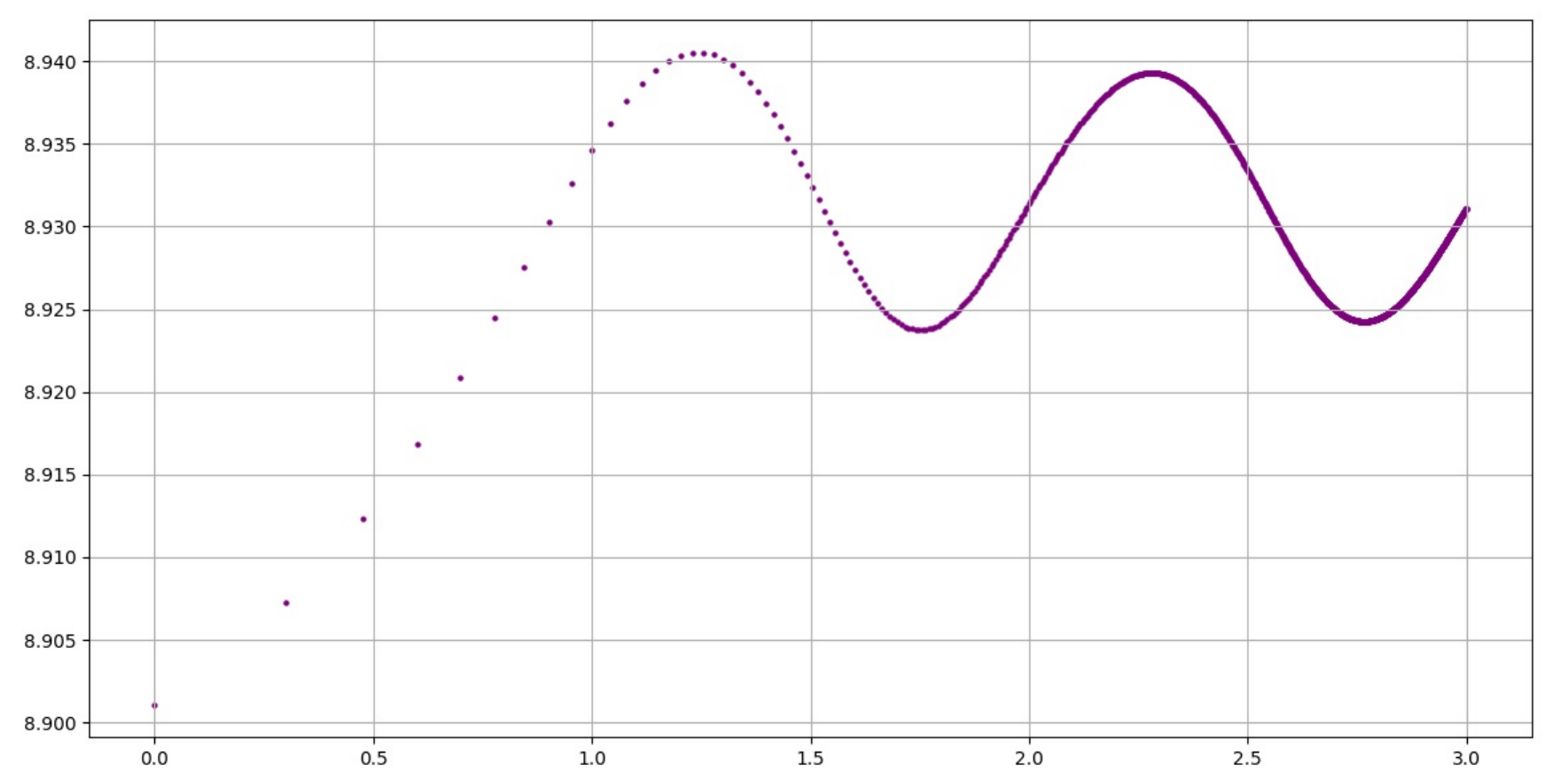}
\end{figure}
This clearly suggests that as a function of $\log_{10}(m)$, the (rescaled)
moments are asymptotically $1$-periodic, a fact which may not have been
immediately obvious from the defining Equation \eqref{eq:intro1}:
$u_{10m}\approx (\frac{9}{8})^{-9m} 10^{-1}u_m$.

On second thoughts, we can understand on a simplified case the mechanism.
Instead of making
illegal the digit $9$, we, to the contrary, allow only it.  So now we are
looking at much simpler quantities:
\begin{equation}
  \label{eq:intro2}
    u_m = 0^m + \sum_{l=1}^\infty \bigl(1 - 10^{-l}\bigr)^m 10^{-l}.
\end{equation}
We can if we like expand the power via the binomial formula and obtain some
explicit finite expression for these new $u_m$'s, but the obtained
formula is unstable numerically, having very large (relative to the full
expression) individual contributions which almost cancel one another.  The
recurrence relation (see Equation \eqref{eq:umrec} in general) is stable
numerically, but does not look promising for the study of the asymptotic.
Rather, we examine more closely the original sum in Equation
\eqref{eq:intro2}.  We see that for large $m$, the initial terms are
exponentially small. Only when reaching $l$ of the order of $\log_{10}(m)$ do
we start having substantial contributions (of order $m^{-1}$ due to
$10^{-l}$). But at this stage, the power $(1 -10^{-l})^m$ will not differ much
from $\e^{-m 10^{-l}}$, and we thus expect that our $u_m$ will not differ much
from
\begin{equation}\label{eq:intro3}
  g(m) =  \sum_{l=0}^\infty \frac{\e^{-m10^{-l}}}{10^l}\,.
\end{equation}
Looking at the defining series for this $g(m)$, we see that adding terms with
negative $l$'s will hardly modify its value, as they are
exponentially small. So we now consider
\begin{equation}\label{eq:intro4}
  f(m) =  \sum_{l=-\infty}^\infty \frac{\e^{-m10^{-l}}}{10^l}\,,
\end{equation}
and observe that $mf(m)$ is invariant under the replacement of $m$ by $10m$.
Such functions as $g(m)$ and $f(m)$ have a long history in the literature
going all the way back to Hardy \cite{hardy1907}.  Here are some relevant
papers \cite{balamendsebb2005,keatread2000,mendsebb1999,zhang2011,zhang2023}
from a large literature.

As is to be expected, the general theory for the quantities defined by
Equation \eqref{eq:intro1} proves a bit more involved than the above sketch given
for the ``toy-model'' Equation \eqref{eq:intro2} (for graphics related to
this ``toy-model'', see Figures \ref{fig:1}, \ref{fig:2}, and \ref{fig:3}). In
addition to the appearance in the Fourier analysis of values of the Gamma
function on vertical lines in the complex plane, which are already present in
the ``toy-model'', the general case will also involve values of Hurwitz zeta
functions with missing digits.  At $s=1$, such series of the type $\sum'
(n+h)^{-1}$ with restrictions on the radix $b$ representation of the allowed
integers $n$ already arose a long time ago in the work of Fischer
\cite{fischer}.  Interestingly enough, the restriction on the digits is not
the same (in general) as the one which is imposed upon the original Riemann
zeta series $\sum' n^{-s}$ with missing digits: one now allows exactly those
digits $d=f-a$ where $f$ was the largest originally allowed
digit, and the $a$'s are the originally allowed ones.

Here is how the contents are organized: we start with a general Lemma on
moments of measures on the unit interval with a H\"older condition at $1$; then
we prove the asymptotic periodicity in $\log_b(m)$ (where $b$ is the used
radix) of suitably rescaled moments, in complete generality (Theorems
\ref{thm:main} and \ref{thm:main2}).  We use some results from our earlier
paper \cite{burnolzeta}, particularly those from its section \S3 which give
the asymptotic behavior of the exponential generating function of the moments.
But to a large extent our exposition here remains self-contained, for the
convenience of the reader. The last section of the paper is dedicated to
various examples, it contains nine figures and focuses on confirming the
expected numerical value of the asymptotic average.

\section{A lemma on moments}

This proposition is needed to establish Theorems \ref{thm:main} and
\ref{thm:main2} in the next section.  It is probably classical (and may even
be a textbook exercise, at least for the case of a measure with a continuous
density function) but the author could not locate a reference.
\begin{prop}\label{prop:1}
  Let $\mu$ be a complex  measure on (the Borel algebra of) $\Ifo01$, which
  has the following H\"older property at $1$, for some $\sigma\geq0$:
  \begin{equation}\label{eq:Hs}\tag{$H_\sigma$}
    \exists K<\infty\qquad 0\leq x<1\implies |\mu|\bigl(\Ifo x1\bigr) \leq K(1-x)^\sigma.
  \end{equation}
  Let, for $m$ real and non-negative,
\[
u_m = \int_{\Ifo 01} x^m \mu(\dx)
\]
  be
  the (general) moments, and let
\[
E(t)= \int_{\Ifo 01}\e^{tx}\mu(\dx)
\]
 be the
  exponential generating function of the (integer) moments.  There holds:
  \begin{equation*}
    u_m = \e^{-m} E(m) + O_{m\to\infty}(\frac1{m^{1+\sigma}}).
  \end{equation*}
\end{prop}
\begin{proof}
  There holds
\begin{equation*}
  \Bigl|\e^{-m} E(m) - u_m\Bigr| \leq \int_{\Ifo 01} \left| \e^{-m+mx}-x^m\right||\mu|(\dx)
= \int_{\Ifo 01} \left( \e^{-m+mx}-x^m\right)|\mu|(\dx),
\end{equation*}
as the integrand in the right-hand side is non-negative (even positive for
$m>0$). We are thus reduced to the case of a finite positive measure.

Let $\mu$ be such a positive measure, verifying \eqref{eq:Hs} for some $\sigma\geq0$.
There holds $u_m = O(2^{-m}) +
\int_{\Ifo{\frac12}1}x^m\mu(\dx)$ and,
also splitting at $\frac12$, we have
$\e^{-m}E(m)=
\int_{\Ifo01}\e^{-m(1-x)}\mu(\dx)= O(\e^{-m/2}) +
\int_{\Ifo{\frac12}1}\e^{-m(1-x)}\mu(\dx)$.  Thus, the claim is reduced to
\begin{equation}
  \label{eq:1}
  \int_{\Ifo{\frac12}{1}}\bigl(\e^{-m(1-x)}- x^m\bigr)\mu(\dx) = O(\frac1{m^{1+\sigma}}).
\end{equation}
There holds for any positive real number $m$ and $x\in \Ifo{\frac12}1$:
  \begin{equation}
    \label{eq:2}
    0< \e^{-m(1-x)} - x^m < m(1-x)^2 \e^{-m(1-x)}.
  \end{equation}
  Let us give a proof for this textbook inequality.
  Let $h=1-x$, $0<h\leq \frac12$. We want
\begin{equation*}
  0< 1 - (1-h)^m\e^{mh} < mh^2.
\end{equation*}
The middle term is $1 - e^{-y}$ with
\[
y = - m \log(1-h) - mh
= m\Bigl(\frac12h^2+\frac13h^3+\dots\Bigr)<m \frac{h^2}{2(1-h)}\leq m h^2.
\]
Hence $0<y<mh^2$ and $0<1 - \e^{-y}<y<mh^2$.  This proves \eqref{eq:2}.

We are now reduced to show:
\begin{equation}
  \label{eq:3}
  \int_{\Ifo{\frac12}{1}}(1-x)^2\e^{-m(1-x)} \mu(\dx)= O(\frac{1}{m^{2+\sigma}}).
\end{equation}
Let $\psi_m(x)=\ddx \left((1-x)^2\e^{-m(1-x)}\right) = \bigl(m(1-x)-2\bigr)(1-x)\e^{-m(1-x)}$.
Up to an exponentially small term, the left-hand side of \eqref{eq:3} is
\begin{equation*}
  \int_{\Ifo{\frac12}{1}}\left((1-x)^2\e^{-m(1-x)}-\frac{\e^{-m/2}}4\right) \mu(\dx)
= \int_{\Ifo{\frac12}{1}}\Bigl(\int_{\frac12}^x\psi_m(y)\dy\Bigr)\mu(\dx).
\end{equation*}
Using Fubini and the hypothesis \eqref{eq:Hs}:
\begin{equation*}
\left| \int_{\Ifo{\frac12}{1}}\Bigl(\int_{\frac12}^x\psi_m(y)\dy\Bigr)\mu(\dx)\right |
= 
  \left|\int_{\frac12}^1 \psi_m(y)\mu(\Ifo y1)\dy\right| 
\leq K \int_{\frac12}^1|\psi_m(y)|(1-y)^\sigma\dy.
\end{equation*}
The left-hand side of \eqref{eq:3} is thus bounded by the sum of some
$O(\e^{-m/2})$ with
\begin{equation*}
  K \int_0^{\frac12}|mt-2| t^{1+\sigma}\e^{-mt}\dt 
= K m^{-2-\sigma}\int_0^{\frac{m}2}|u-2|u^{1+\sigma}\e^{-u}\du,
\end{equation*}
and the proof of \eqref{eq:3}, hence of the Proposition, is complete.
\end{proof}

\section{Missing digits and associated moments}

Let $b>1$ be an integer and let $A\subset\{0,\dots,b-1\}$ be a non-empty
subset of the set of radix-$b$ digits, of cardinality $N$.  Let $f=\max A$.
We assume $f>0$.

In order to simplify the presentation, we also assume provisionally:
\begin{equation*}
  f = b-1.
\end{equation*}
The additional discussion needed if $f<b-1$ will be provided only at the end
of this section after the proof of Theorem \ref{thm:main}.

We say that a positive integer is \emph{$(b,A)$-admissible}, in short
\emph{admissible} if its radix-$b$ representation uses only digits from $A$.
As zero is represented by the empty word, we consider it also to be
admissible (whether or not the digit $0$ is in $A$).

We want to discuss further some properties of the moments
$u_m(s)$ which are defined in \cite[\S2]{burnolzeta} as
\begin{equation*}
  u_m(s) = \int_{[0,1)}x^m \mu_s(\dx),
\end{equation*}
where $\mu_s$ is some complex measure and $s$ is a complex number of real part
greater than $s_0 = \log_bN$. The discrete complex measure $\mu_s$ is defined
on the Borel algebra of $\Ifo01$ in the following manner: first, we put a
Dirac unit mass at zero. Then we add Dirac masses with (complex) weights
$b^{-s}$ at each $ab^{-1}$, $a\in A$. Then we add Dirac masses with weights
$b^{-2s}$ at each $(a_1b+a_2)b^{-2}$, $(a_1,a_2)\in A^2$, and so on with
$A^l$, $l\to\infty$.  This defines a complex measure whose zeroth moment
$\mu_s\bigl(\Ifo01\bigr)$ is $b^s/(b^s-N)$ (\cite[\S2]{burnolzeta}).
Indeed:
\begin{equation*}
  1 + \sum_{l=1}^\infty \sum_{(a_1,a_2,\dots,a_l)\in A^l} b^{-ls}
=
  \sum_{l=0}^\infty N^lb^{-sl} 
= \frac{b^s}{b^s-N}\;.
\end{equation*}
We need the following:
\begin{lemma}\label{lem:1}
  The complex measure $\mu_s$ on $\Ifo 01$ verifies hypothesis \eqref{eq:Hs}
  from Proposition \ref{prop:1}, with $\sigma=\Re s$.
\end{lemma}
\begin{proof}
  Let $k\geq0$, we evaluate $\mu_s\bigl(\Ifo{1-b^{-k}}1\bigr)$ (we will need
  the result only for $s$ real, but it is worthwhile to prove it in general).
  Recall that we assume $b-1\in A$. For $l\geq0$ and $(a_1,\dots,a_l)\in A^l$
  to contribute weight to the interval $\Ifo{1-b^{-k}}1$, it is necessary and
  sufficient that
  \begin{equation*}
    \frac{a_1}b + \dots + \frac{a_l}{b^l} \geq 1 - b^{-k} = 
    \frac{b-1}b + \dots + \frac{b-1}{b^k}\;.
  \end{equation*}
  This happens if and only if $l\geq k$ and $a_1 = \dots = a_k = b-1$. The
  digits $a_q$ for $k<q\leq l$ are thus arbitrary in $A$.  The weight
  contributed by $(b-1,\dots,b-1,a_{k+1},\dots,a_l)$ is $b^{-ks}$ times the
  weight which $(a_{k+1},\dots,a_l)$ contributes to the mass of $\Ifo01$.  So
  we obtain
\begin{equation}\label{eq:mutail}
  \mu_s\bigl(\Ifo{1-b^{-k}}1\bigr) = b^{-ks}\mu_{s}\bigl(\Ifo 01\bigr).
\end{equation}
Let $\sigma=\Re s$. There holds $|\mu_s|=c_s\mu_\sigma$ with $c_s=1$ if $0\notin A$ and $c_s=(1-b^{-\sigma})/|1-b^{-s}|$ if $0\in A$.
Equation \eqref{eq:mutail} applied with $s=\sigma$ says that
$\mu_{\sigma}\bigl(\Ifo {y_k}1\bigr) = K(1-y_k)^\sigma$ for some $K$ and any
$y_k=1-b^{-k}$, $k\geq0$.  Let $0\leq x<1$ and $k$ be the largest such that
$y_k\leq x$. Then $1-y_k\geq 1-x>\frac{1-y_k}{b}$, hence $\mu_\sigma\bigl(\Ifo
x1\bigr)\leq K(1-y_k)^\sigma\leq Kb^\sigma (1-x)^\sigma$, which completes the
check of \eqref{eq:Hs}.
\end{proof}

We are interested in the moments of the measure $\mu_s$:
\begin{equation}\label{eq:umdef}
  u_m(s) = 0^m + \sum_{l=1}^\infty 
   \sum_{(a_1,a_2,\dots,a_l)\in A^l} \Bigl(a_1b^{-1}+\dots +a_lb^{-l}\Bigr)^m b^{-ls}.
\end{equation}
Their precise numerical evaluation is not easy, except if $\# A=N = 1$, due,
if $N>1$, to exponentially many contributions to each term.  This can be
mitigated to some extent, for integer moments, via the recurrence formula from
\cite[Prop.\@ 2.5]{burnolzeta}:
\begin{equation}\label{eq:umrec}
m\geq1\implies
  u_m(s) =
  \frac1{b^{m+s} - N}\sum_{j=1}^m \binom{m}{j}\Bigl(\sum_{a\in A} a^j\Bigr) u_{m-j}(s).
\end{equation}
The proof of the main theorem will not use this recurrence relation though.
\begin{thm}\label{thm:main}
  Let $B$ be the set $f-A = \{b-1 - a, a \in A\}$ (we assume $\max A=b-1$) and let
  $\cB$ be the set of $B$-admissible (non-negative) integers.  There holds for
  $m$ real going to $+\infty$:
\begin{align}\notag
m^s u_m(s)
={}
&\underbrace{(\log b)^{-1}\sum_{k=-\infty}^\infty 
  \e^{2\pi i k \log_b(m)}
\Gamma(s - 2\pi i \frac k{\log b}) 
\sum_{n\in\cB}\frac{1}{(n + 1)^{s-2\pi i \frac k{\log b}}}}
_{F_s(m)}
\\\label{eq:main1}
&+ O(\frac{1}{m}),
\end{align}
which represents $m^s u_m(s)$ as the sum of a smooth $1$-periodic function of
$\log_b(m)$ and an $O(m^{-1})$ error term.  So, the sequence of integer moments
$(m^s u_m(s))_{m\geq0}$ is asymptotic to a sequence $(F_s(m))$ invariant under
$m\mapsfrom bm$, which can also be written as
\begin{equation}\label{eq:main2}
F_s(m) = m^s \sum_{j=-\infty}^\infty b^{js}\e^{-b^j m}\sum_{n\in \cB}\e^{-nb^j m}.
\end{equation}
The asymptotic average value for the sequence of integer moments is:
\begin{equation}
  \label{eq:main3}
  \langle m^su_m(s)\rangle \coloneq \lim_{p\to\infty}\sum_{b^p<m\leq b^{p+1}}m^su_m(s)\log_b{\frac{m}{m-1}} = \frac{\Gamma(s)}{\log b}\sum_{n\in \cB}\frac{1}{(n+1)^s}\;.
\end{equation}
\end{thm}
\begin{proof}
  According to Lemma \ref{lem:1} and Proposition \ref{prop:1}, $m^s
  u_m(s) = m^s \e^{-m}E(m) + O(m^{-1})$, so the matter is reduced to the
  asymptotic of $m^s \e^{-m}E(m)$.  But this is exactly what was studied in
  \cite[Prop.\@ 3.3]{burnolzeta}.  To be self-contained we explain again
  the simple steps.  Define first
  \begin{equation*}
    \alpha_B(t) = \sum_{d\in B}\e^{dt}.
  \end{equation*}
Observe that $0<\alpha_B(t)\leq N$ for $t\leq 0$.
  The moment generating function is (for $m$ real positive):
  \begin{align*}
E(m) &=  1 + \sum_{l=1}^\infty 
   \sum_{(a_1,a_2,\dots,a_l)\in A^l} \e^{m(a_1b^{-1}+\dots +a_lb^{-l})} b^{-ls}\\
&=1 + \sum_{l=1}^\infty 
   \sum_{(d_1,d_2,\dots,d_l)\in B^l} \e^{m(1 - b^{-l}-d_1b^{-1}-\dots -d_lb^{-l})} b^{-ls},\\
\e^{-m}E(m) &= \sum_{l=0}^\infty 
   \bigl(\prod_{1\leq k\leq l}\alpha_B(-mb^{-k})\bigr) \e^{- mb^{-l}}b^{-ls}.
  \end{align*}
  Note that for $m\geq 0$, this series is bounded term-wise by $N^l b^{-l\Re s}$
  (as it should be from where it originates from) and convergence is indeed
  ensured.  We now introduce the infinite product
\begin{equation*}
  \gamma_B(-t) = \prod_{i=0}^\infty \alpha_B(-b^it).
\end{equation*}
Here it is important that $0\in B$ so $\alpha_B(-b^i t) = 1 + O(\e^{-b^iD t})$
where $D$ is the smallest positive element of $B$ (if one such exists, else
$\alpha_B$ is the constant $1$) and the product is absolutely convergent for
$\Re t>0$. We will use it only with $t$ real positive and certainly
$\gamma_B(-t) = 1 + O_{t\to+\infty}(\e^{-Dt})$ (or the constant $1$ if
$B=\{0\}$).  Looking closer we recognize that the infinite product has a
series expansion in powers of $\e^{-t}$:
\begin{equation*}
  \gamma_B(-t) = \sum_{n\in \cB} \e^{-nt},
\end{equation*}
where $\cB$ is the set of $B$-admissible non-negative integers (which is reduced
to the singleton $\{0\}$ if $B=\{0\}$, else has $D$ as minimal positive element).
We then obtain the formula:
\begin{equation*}
  \gamma_B(-m)\e^{-m}E(m) =  \sum_{l=0}^\infty \gamma_B(-mb^{-l})\e^{-mb^{-l}}b^{-ls}.
\end{equation*}
This leads to the definition of the $1$-periodic (in $\log_b(m)$) smooth function:
\begin{equation*}
  F_s(m) = m^s\sum_{l=-\infty}^\infty \gamma_B(-b^{-l}m)\e^{-mb^{-l}}b^{-ls},
\end{equation*}
where the added terms for negative $l$ are exponentially small as functions of $m$.
Hence
\begin{equation*}
  m^s\e^{-m}E(m) = (1 - O(\e^{-Dm}))\Bigl(F_s(m) - O(m^{\Re s}\e^{-mb})\Bigr).
\end{equation*}
So $m^s\e^{-m}E(m)$ differs from $F_s(m)$, which is invariant under
$m\mapsfrom bm$, by an exponentially small function of $m\to+\infty$.  We can
represent $F_s(m)$ as a Fourier series in the variable $\log_b(m)$. The
computation of the Fourier coefficients is given in \cite[proof of Prop.\@
3.3]{burnolzeta} and gives the underbraced expression in \eqref{eq:main1}.

So we have established \eqref{eq:main1} and \eqref{eq:main2}.  As per the fact
that the middle term of \eqref{eq:main3} converges to the zeroth Fourier
coefficient whose value is given in \eqref{eq:main1} and gives the third term,
it is a simple exercise of Riemann sums, which is left to the reader.  We can
use $\log_b(1+\frac1m)$ in place of $\log_b(m/(m-1))$ if one prefers, but
preferably then summing from $b^p$ to $b^{p+1}-1$ so that this is really a
barycentric average: this is better numerically.
\end{proof}

We now lift the restriction on the value of $f=\max A$.  Let
$\kappa=\frac {b-1}f$. In place of directly $u_m(s)$ we want to consider the
rescaled $\kappa^mu_m(s)$ which are the moments of the measure obtained by
keeping the same individual weights as in $\mu_s$ but multiplying all rational
numbers in its support by $\kappa$, so that their supremum is now at $1$ and not
$f/(b-1)$. Let $\nu_s$ be this new complex measure. We leave to the reader to
check that this $\nu_s$ verifies \eqref{eq:Hs} from Proposition \ref{prop:1}
with $\sigma=\Re s$.  Hence $\kappa^m m^s u_m(s)$ differs by $O(m^{-1})$ from
$m^s\e^{-m}$ times the generating function for the $\nu_s$ moments, evaluated
at $m$, which is $E(\kappa m)$.  We are thus reduced to study $m^s
\e^{-m} E(\kappa m)$, where $E$ is the generating function of the moments of
$\mu_s$. And:
  \begin{align*}
E(\kappa m) &=  1 + \sum_{l=1}^\infty 
   \sum_{(a_1,a_2,\dots,a_l)\in A^l} \e^{\kappa m(a_1b^{-1}+\dots +a_lb^{-l})} b^{-ls}\\
&=1 + \sum_{l=1}^\infty 
   \sum_{(d_1,d_2,\dots,d_l)\in B^l} \e^{\kappa m \bigl( (1- b^{-l})\frac f{b-1} -d_1b^{-1}-\dots -d_lb^{-l}\bigr)} b^{-ls},\\
m^s\e^{-m}E(\kappa m) &= \sum_{l=0}^\infty 
   \bigl(\prod_{1\leq k\leq l}\alpha_B(-\kappa mb^{-k})\bigr) \e^{- mb^{-l}}(mb^{-l})^{s},\\
m^s\gamma_B(-\kappa m)\e^{-m}E(\kappa m)&=
\sum_{l=0}^\infty 
   \gamma_B(-\kappa mb^{-l})\e^{- mb^{-l}}(mb^{-l})^{s}.
  \end{align*}
  In the last equations we use the set $B=f-A$ from the statement of Theorem
  \ref{thm:main}, and the functions $\alpha_B$ and $\gamma_B$ as defined in
  its proof.  It is important that $0\in B$.

  In order to match the notations from \cite[\S3, eq.(21)]{burnolzeta} we
  now define:
\begin{equation*}
  F_s(t) = \sum_{l=-\infty}^\infty \gamma_B(-tb^{-l})\e^{- \frac{f}{b-1}tb^{-l}}(t b^{-l})^{s}.
\end{equation*}
This function is $1$-periodic in $\log_b(t)$. Similarly to the earlier situation,
when we had $\kappa=1$, we observe that $m^s\e^{-m}E(\kappa m)$ is exponentially close to
$\kappa^{-s}F_s(\kappa m)$.  So $\kappa^m m^s u_m(s)$ is $O(m^{-1})$ close to
the periodic function $\kappa^{-s}F_s(\kappa m)$.  Taking now into account
the Fourier series given in \cite[Prop.\@ 3.3]{burnolzeta} we end up with this
more general version of Theorem \ref{thm:main}.
\begin{thm}\label{thm:main2}
  Let $A\subset\{0,\dots,b-1\}$, $f=\max A>0$, $B = \{f - a, a \in A\}$ and
  $\cB$ be the set of $B$-admissible (non-negative) integers.  Let
  $\kappa=\frac{b-1}f$. Then for $m$ real going to
  $+\infty$, $\kappa^m m^s u_m(s)$ differs by $O(m^{-1})$ from a smooth
  $1$-periodic function in $\log_b(m)$ which is given as
\begin{equation*}
    (\log b)^{-1}\sum_{k=-\infty}^\infty 
    \e^{2\pi i k \log_b(m)}
    \Gamma(s - 2\pi i \frac k{\log b})
    \sum_{n\in\cB}\frac{1}{(\kappa n + 1)^{s-2\pi i \frac k{\log b}}}\;,
\end{equation*}
or equivalently
as
\begin{equation*}
m^s \sum_{j=-\infty}^\infty b^{js}\sum_{n\in \cB}\e^{-(\kappa n+1)b^j m}.
\end{equation*}
The asymptotic average value for the sequence of scaled integer moments is:
\begin{equation*}
  \langle \kappa^m m^su_m(s)\rangle =
  \frac{\Gamma(s)}{\log b}\sum_{n\in \cB}\frac{1}{(\kappa n+1)^s}\;.
\end{equation*}
\end{thm}
In the next section we give some numerical results which all proved to be in
accordance with Theorems \ref{thm:main} and \ref{thm:main2}.

\section{Numerical and graphical examples}

\noindent\textbf{Example 1.} We suppose $N=1$, so $A= \{b-1\}$ and we take $s=1$. So:
\begin{equation*}
  u_m(1) = 0^m + \sum_{l=1}^\infty (1 -b^{-l})^m b^{-l}.
\end{equation*}
Theorem \ref{thm:main} says that $(mu_m(1))$ is asymptotic to a sequence invariant under
$m\mapsfrom bm$, with oscillations around an average value equal to
$(\log b)^{-1} \sum_{n\in\cB}\frac1{n+1}$. But here $B = \{0\}$ and the sole
$B$ admissible integer is $0$.  Hence this average value is $(\log
b)^{-1}$. The asymptotic periodicity as function of $\log_b(m)$ and value of
the average are well confirmed numerically.  It is known (see
\cite{burnolkempner} or \cite{burnolzeta}) that $u_m(1)=O(m^{-1})$, so the
$O(m^{-1})$ error term from Equation \eqref{eq:main1} hides the difference
between $mu_m(1)$ and $(m+1)u_m(1)$ but the latter is seen to fit better the
asymptotic in the sense that the difference with $F_1(m)$, multiplied by $m$,
is seen numerically to average to zero (but definitely has itself $O(1)$
oscillations, so the $O(m^{-1})$ from Equation \eqref{eq:main1} cannot be
replaced by a smaller big-O, even if replacing $mu_m(1)$ by $(m+1)u_m(1)$).

Figures \ref{fig:1}, \ref{fig:2}, \ref{fig:3} are respectively with
$b=2$, $b=3$ and $b=10$.  The code evaluated at all integers in the respective
ranges, which was overkill and the scatter plots became continuous. One can
obtain more economically many more periods by evaluating say at only 100
points per period.  This is feasible because the defining series are usable
directly not being hampered by exponential explosion. (The number of terms
for a given fixed target precision is $\leq \log_b(m)+O(1)$ so we can
compute easily for very large individual $m$'s).

\begin{figure}[htbp]
   \caption{$b=2$, $A=\{1\}$, $(m+1)u_m$ at $s=1$ vs $\log_2(m)$, $20\leq m\leq \np{12800}$}
\label{fig:1}
  \centering
\includegraphics[width=\linewidth]{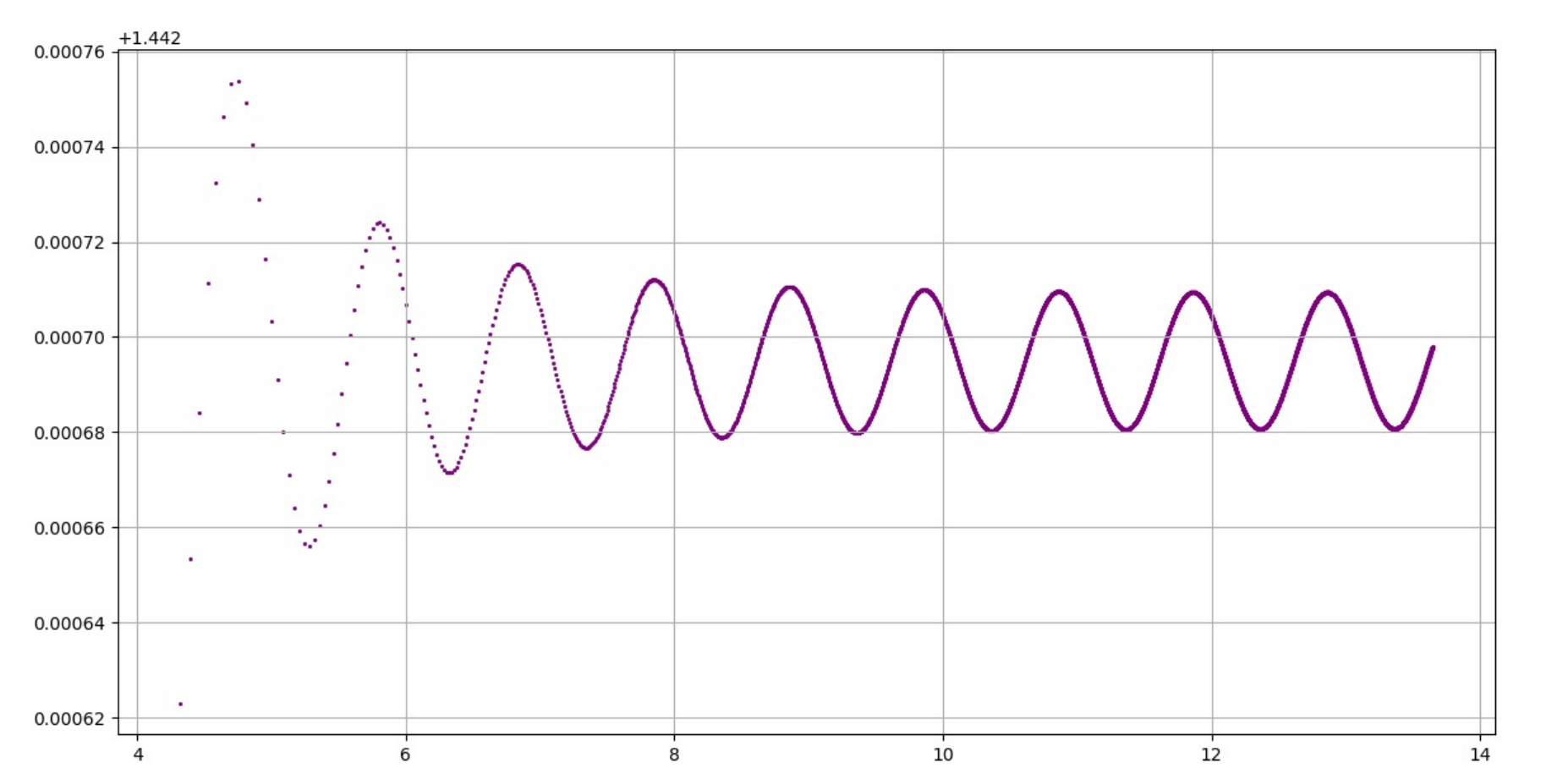}
\end{figure}
\begin{figure}[htbp]
  \caption{$b=3$, $A=\{2\}$, $(m+1)u_m$ at $s=1$ vs $\log_3(m)$, $20\leq m\leq \np{50000}$}
\label{fig:2}
\centering
\includegraphics[width=\linewidth]{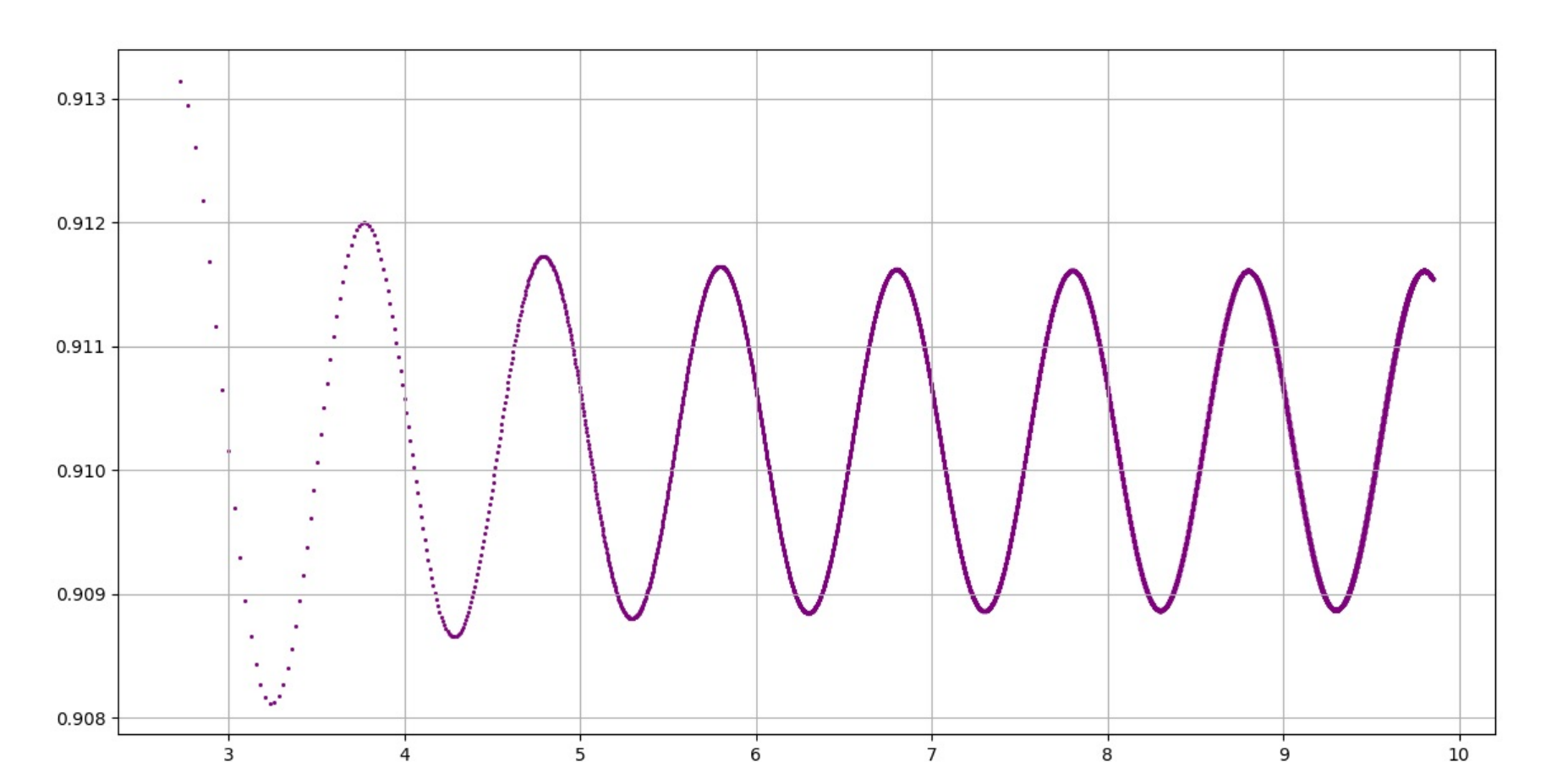}
\end{figure}
\begin{figure}[htbp]
  \caption{$b=10$, $A=\{9\}$, $(m+1)u_m$ at $s=1$ vs $\log_{10}(m)$, $20\leq m\leq \np{100000}$}
\label{fig:3}
\centering
\includegraphics[width=\linewidth]{ 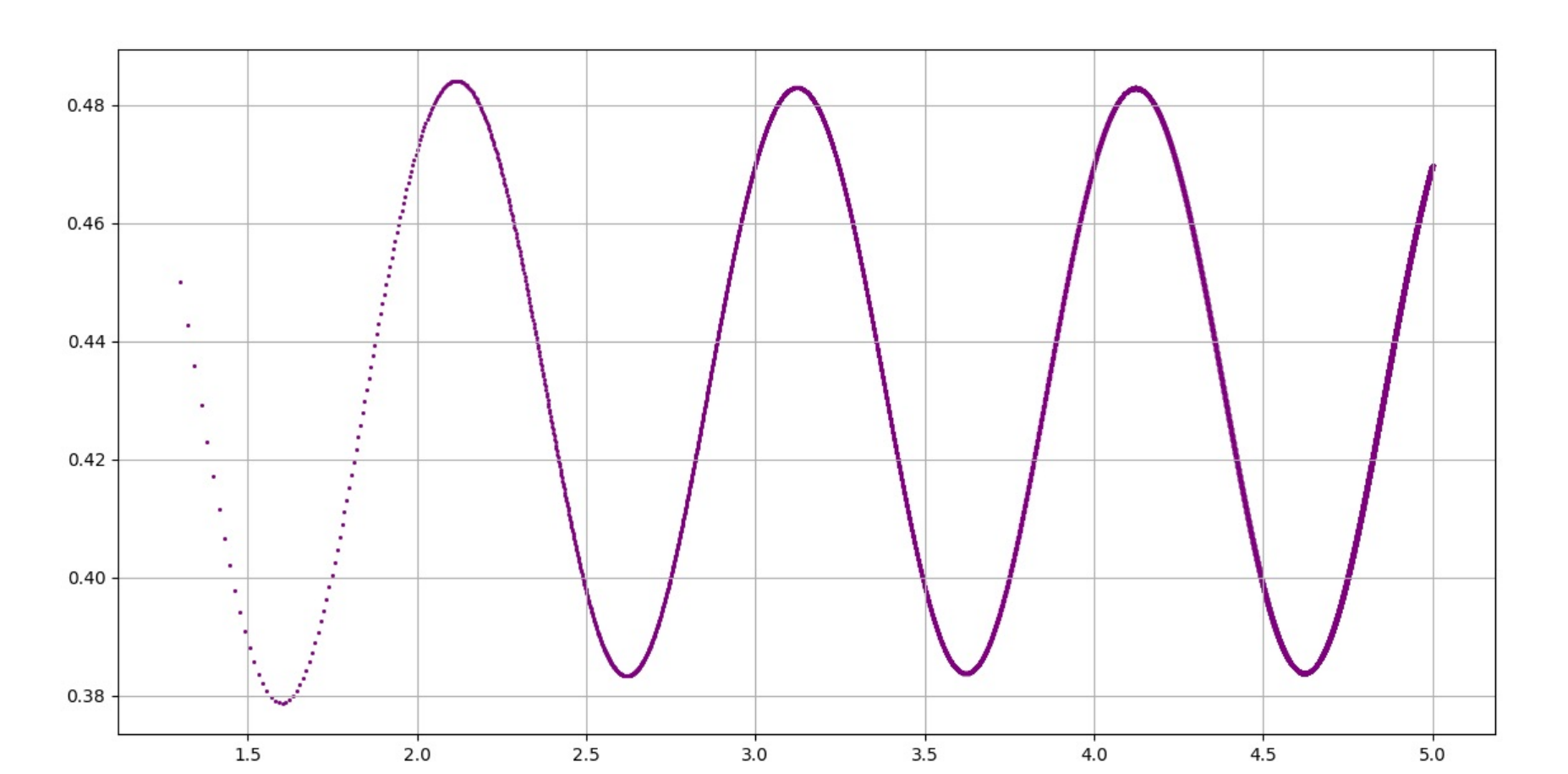}
\end{figure}

\bigskip

\noindent\textbf{Example 2.} We take $b=3$ and $A=\{0,2\}$, so $N=2$, and $B=A$.
Again we take $s=1$. Our main Theorem says that the $(mu_m(1))$ sequence will
asymptotically oscillate around a value equal to $(\log 3)^{-1}$ times the
``Hurwitz-Kempner'' sum $S_1= \sum_{n\in \cA}\frac1{n+1}$.  Here $\cA$ is the
set of non-negative integers using only $0$ and $2$ in base $3$.  We have at
time of writing no (own) immediately available script to compute the shifted Kempner
sum, but we know from using \texttt{kempner.py} available at
\url{https://gitlab.com/burnolmath/kempner} that $S_0= \sum_{n\in \cA,
  n>0}\frac1n$ rounded to 15 decimal places is $\np{1.341426555483088}$
(hundreds, even thousands of decimal digits can be obtained via this script).
We have $S_1 - S_0 = 1 - \sum_{n>0,n\in\cA}\frac1{n(n+1)}$. We can with no
problem let some (e.g. \textsf{Python}) code find all admissible integers less
than say $3^{15}$ (or even somewhat more) and obtain in this way numerically
(using for example \textsf{Python} \texttt{float}-type) a numerical
approximation to $\sum_{n>0,n\in\cA}\frac1{n(n+1)}$.  This brute force
computation gives the last term in $S_1 \approx \np{1.341426555483} + 1 -
\np{0.215160680302}$. Hence $S_1\approx\np{2.12626587518}$, and
$S_1/\log(3)\approx\np{1.935410606}$.  Numerically we found, 
using the recipe from Equation \eqref{eq:main3} and
$\np{9001}\leq m\leq \np{27000}$ that the average was about
$\np{1.935410642}$.  This is indeed a good match.
See Figure
\ref{fig:4} for the graph.

\begin{figure}[htbp]
    \caption{$b=3$, $A=\{0,2\}$, $(m+1)u_m(1)$ vs $\log_3(m)$, $21\leq m\leq \np{27000}$}
  \label{fig:4}
\centering
\includegraphics[width=\linewidth]{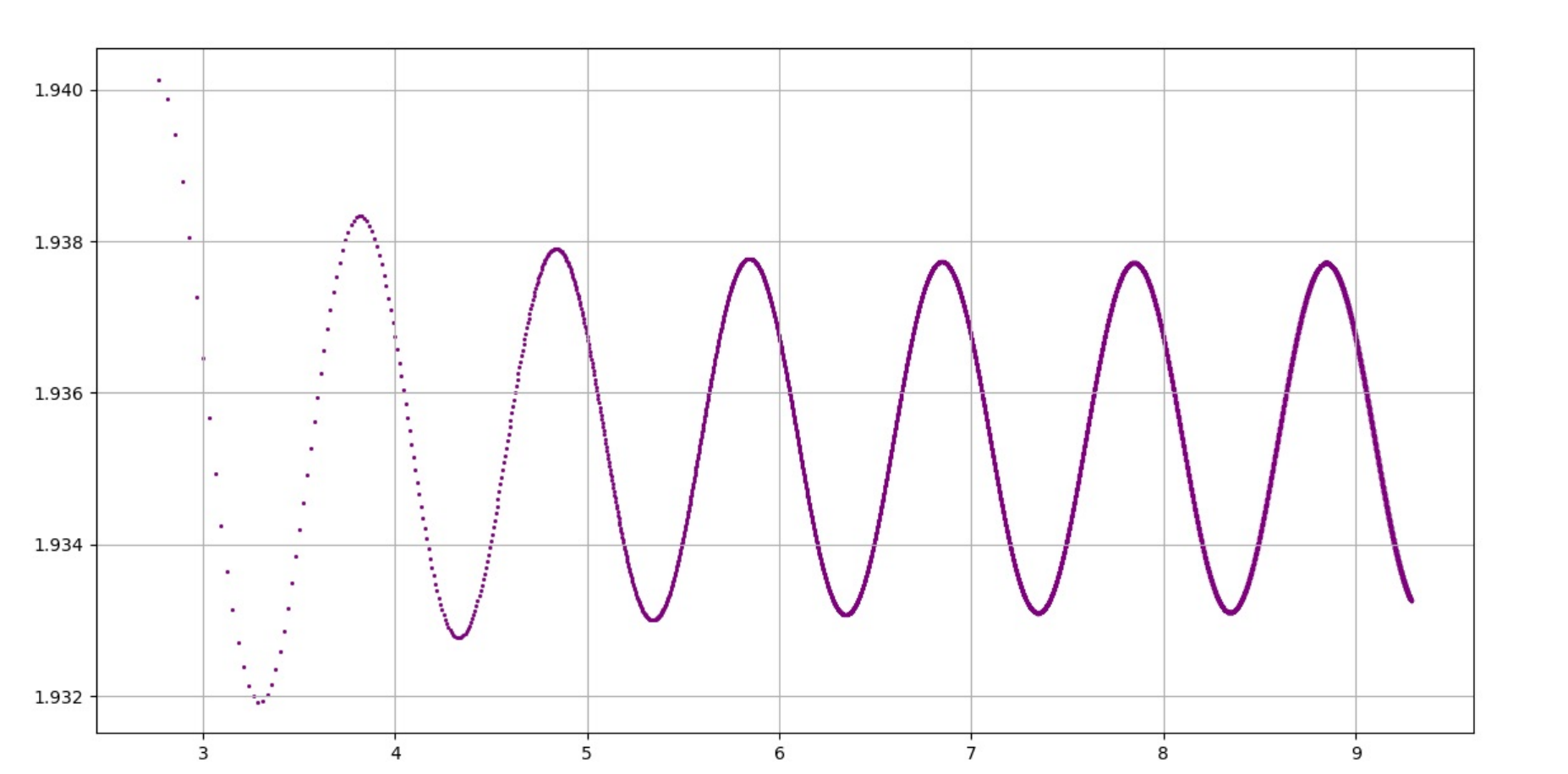}
\end{figure}

\bigskip

\noindent\textbf{Example 3.} We take $b=3$ and $A=\{1,2\}$, so $N=2$, and
$B=\{0,1\}$.  Again we take $s=1$ and adapt the steps of our previous example.
We can approximate $\sum_{n\in\cB} 1/(n+1)$ by $1 + \sum_{\text{``no 2''}}1/n
- \sum_{\text{``no 2''}}1/(n(n+1))$ where the last one is evaluated by brute
force and the middle one is also $2\sum_{\text{``no 1''}}1/n$ hence
known from the previous example or can be obtained using \texttt{kempner.py}.
This gives $1 + 2\times \np{1.34142655548}-\np{0.67491411107}$ for
$\sum_{n\in\cB}(n+1)^{-1}$. We divide by $\log 3$ to obtain about
$\np{2.73794407}$ as value of the average of the oscillations.  Numerically
using the $u_m$'s for $\np{3001}\leq m\leq \np{9000}$ we obtained an average of about
$\np{2.73794486}$ (and $\np{2.73794406}$ using some extrapolation).  So all is
well and is illustrated in Figure \ref{fig:5}.

\begin{figure}[htbp]
    \caption{$b=3$, $A=\{1,2\}$, $(m+1)u_m(1)$ vs $\log_3(m)$, $21\leq m\leq \np{9000}$}
  \label{fig:5}
\centering
\includegraphics[width=\linewidth]{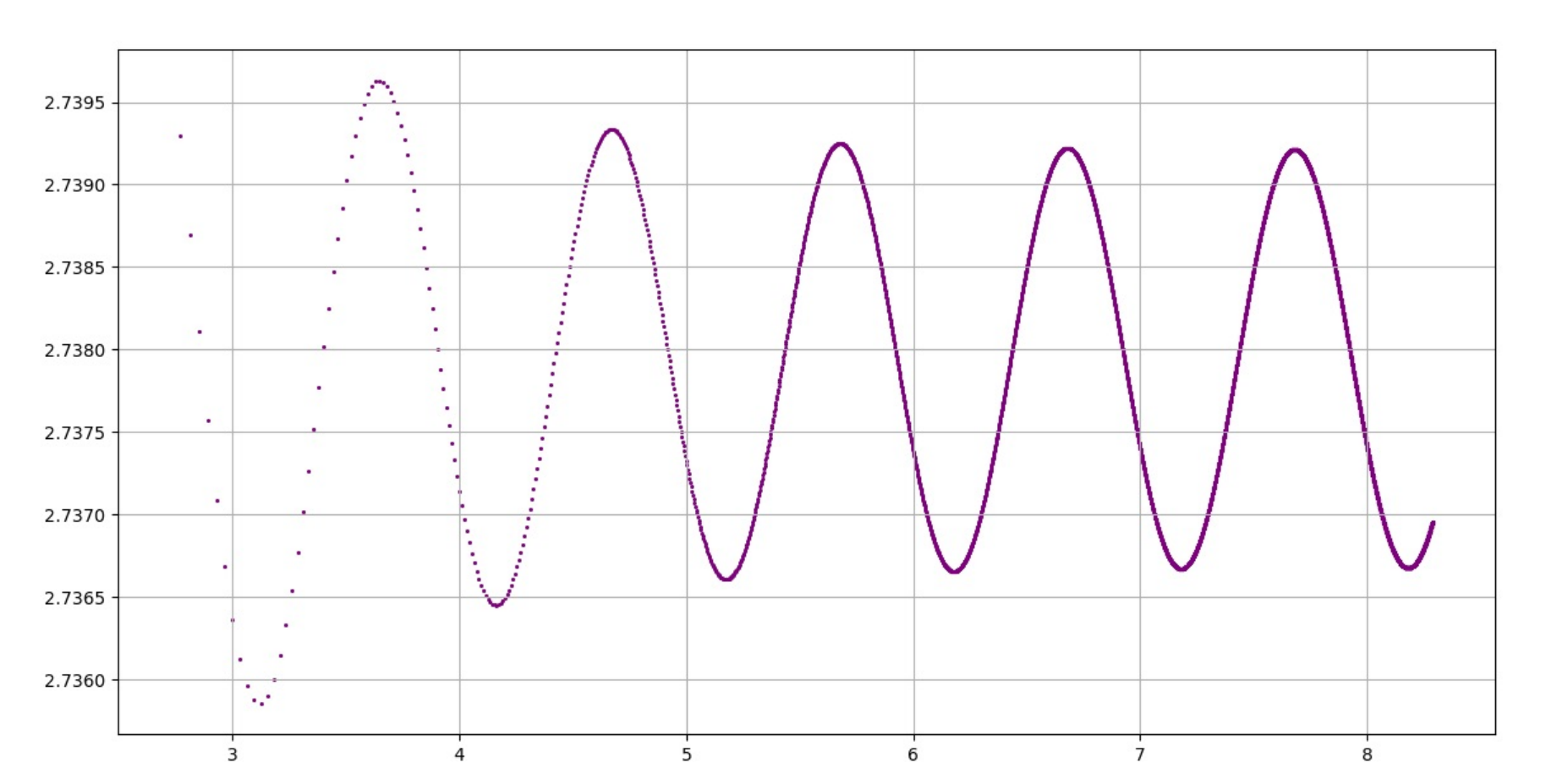}
\end{figure}

\bigskip

\noindent\textbf{Example 4.} We now consider the case with all digits being allowed,
so the $u_m(s)$'s can be used to compute the Riemann zeta function
(\cite{burnolzeta}).  They have an explicit formula in terms of Bernoulli
numbers, for $m\geq1$ (\cite[Thm.\@ 4.1]{burnolzeta}):
  \begin{equation*}
  u_m(s) =  \frac1{m+1}\frac{b^{s}}{b^s -b} - \frac{b^{s+1}}{2(b^{s+1}-b)} +
\sum_{1\leq k \leq \lfloor \frac m2 \rfloor}\frac{m!}{(m-2k+1)!}\frac{B_{2k}}{(2k)!}\frac{b^{s+2k}}{b^{s+2k}-b}\;.
  \end{equation*}
  But intermediate terms in this expression become huge so that the
  computation must use a very elevated precision, and this formula is very
  poor numerically for $m$ going into the thousands (or already much smaller).
  We use in its place the recurrence formula \eqref{eq:umrec}, it is stable
  numerically as it contains only positive contributions ($s$ real).  As was
  the case with $s=1$, it seems that the rescaled $u_m^*(s) =
  \frac{(s+1)_m}{m!} u_m(s)$'s display faster the asymptotic behavior.  As the
  added multiplicative factor is $\sim m^s/\Gamma(s+1)$, the expected average
  numerical value given by Theorem \ref{thm:main} is $\frac{\zeta(s)}{s
    \log(b)}$.  We illustrate this for $s=2$ with $b=2$, $m\leq \np{10000}$, for
  which we found numerically an average about equal to $\np{1.1865690}$ to be
  compared with $\zeta(2)/(2\log(2))\approx \np{1.1865691104156}$.  See Figure
  \ref{fig:6}.  And for $s=3$ with $b=10$, also $m\leq \np{10000}$, we found
  numerically an average about equal to $\np{0.17412609}$ to be compared with
  $\zeta(3)/(3\log(10))\approx \np{0.17401555999}$ (extrapolation from combining
  the average for $\np{1001}\leq m\leq \np{10000}$ with the one for
  $\np{101}\leq m\leq \np{1000}$ gives $\np{0.17401302}$).
   See Figure \ref{fig:7}.
\begin{figure}[htbp]
    \caption{$b=2$, $A=\{0,1\}$, $\frac{(m+1)(m+2)}2 u_m(2)$ vs
      $\log_2(m)$, $16\leq m\leq \np{10000}$}
    \label{fig:6}
\centering
\includegraphics[width=\linewidth]{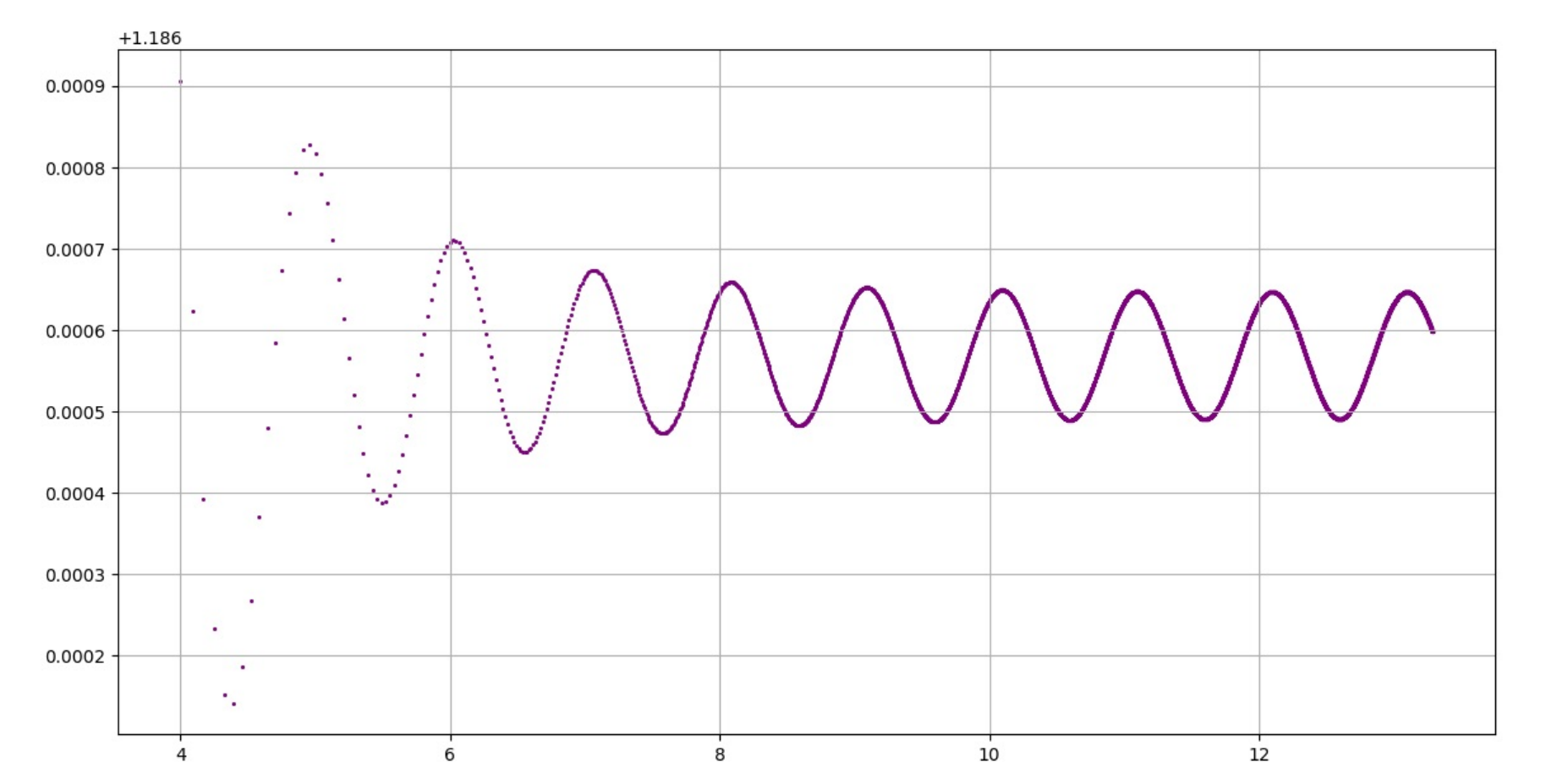}
  \end{figure}

\begin{figure}[htbp]
    \caption{$b=10$, $A=\{0,\dots,9\}$, $\frac{(m+1)(m+2)(m+3)}{6}u_m(3)$ vs $\log_{10}(m)$, $1\leq m\leq \np{10000}$}
    \label{fig:7}
\centering
\includegraphics[width=\linewidth]{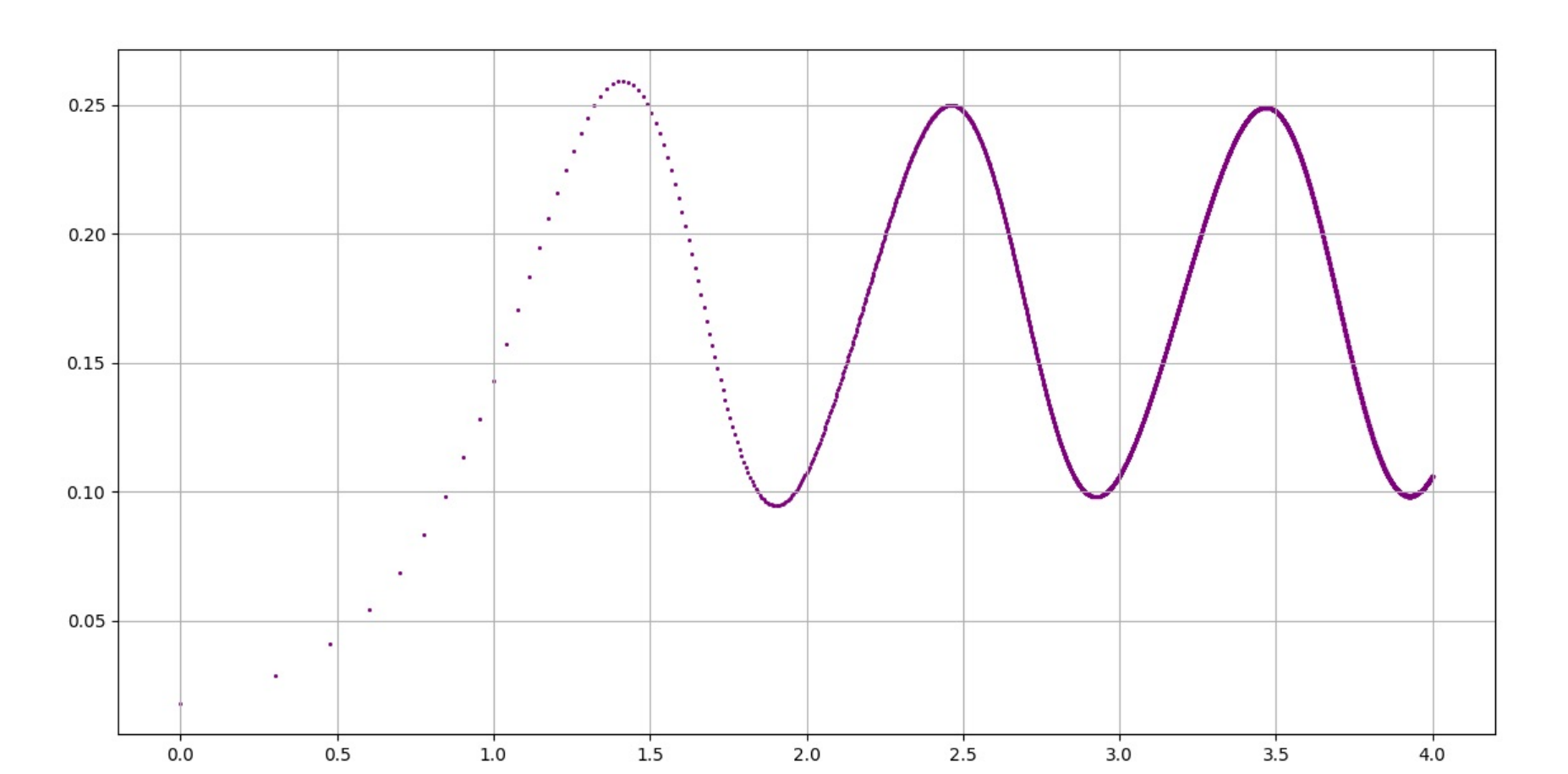}
  \end{figure}

  \bigskip

\noindent\textbf{Example 5.} Our final examples are with $b=8$, $A=\{0,1,3,5\}$
  and $s=1$ or $s=3$.  Here $\kappa=\frac 75$ and we plot
  $\kappa^m(s+1)_m(m!)^{-1}u_m(s)$ for $1\leq m\leq \np{10000}$ as a function of
  $\log_8(m)$ (see Figures \ref{fig:8} and \ref{fig:9}).  As
  $(s+1)_m/m!\sim m^s\Gamma(s+1)^{-1}$, the expected average according to
  Theorem \ref{thm:main2} is $\frac{H(s)}{s \log(8)}$ with $H(s) =
  \sum_{n\in\cB}(7n/5 + 1)^{-s}$, and $\cB$ is the set of non-negative
  integers using in radix-$8$ exclusively the digits from $\{0,2,4,5\}$.  For
  $s=1$, we compute $H(1)$ as the sum of $1+
  \frac57\sum_{n>0,n\in\cB}n^{-1}-\sum_{n>0,n\in\cB}25/(7n(7n+5))$ where the
  last term is evaluated by brute force and the middle one using
  \texttt{kempner.py}.  This gives $H(1)\approx 1 + \np{1.281146045711} -
  \np{0.148015877970}\approx \np{2.133130167741}$.  This gives about
  $\np{1.025818771524}$ after dividing by $\log 8$.  Regarding the numerical
  evaluation for $m\leq \np{10000}$ via the recurrence relation
  \eqref{eq:umrec}, we obtained from it on the last period an
  estimated average of $\np{1.0257836}$. And as the average on the
  previous period was $\np{1.0255340}$,  extrapolating with the hypothesis that
  the error is divided by $8$ at each period gives $\np{1.025819}$ which
  is indeed close to our expectation.

  Regarding $s=3$, we computed approximately $H(3)$ by brute force using a
  simple \textsf{Python} script using only \texttt{float}-type (and with not
  much concern for errors in least significant digits).  This gives
  $H(3)\approx\np{1.0239193028}$. Dividing by $3\log(8)$ we obtain
  $\np{0.16413370005}$.  Numerically, using $m$ from $\np{1251}$ to $\np{10000}$
  we found an average value $\approx\np{0.164113327}$. And an extrapolation
  from comparison with the previous period gave about $\np{0.164136906}$.
\begin{figure}[htbp]
    \caption{$b=8$, $A=\{0,1,3,5\}$, $(\frac75)^m(m+1)u_m(1)$ vs
      $\log_8(m)$, $1\leq m\leq \np{10000}$}
    \label{fig:8}
\centering
\includegraphics[width=\linewidth]{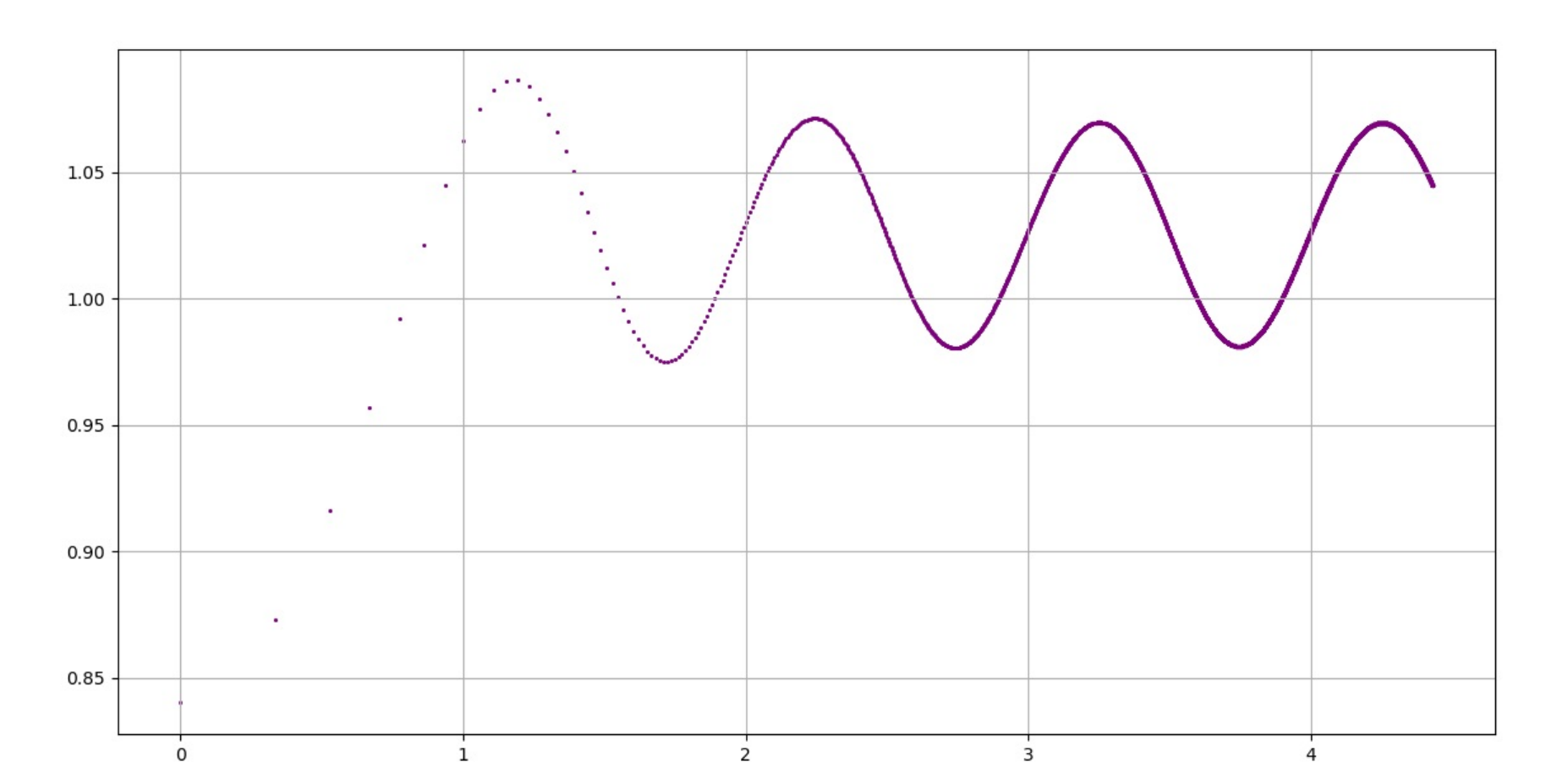}
  \end{figure}

\begin{figure}[htbp]
    \caption{$b=8$, $A=\{0,1,3,5\}$, $(\frac75)^m\binom{m+3}{3}u_m(3)$ vs $\log_{8}(m)$, $1\leq m\leq \np{10000}$}
    \label{fig:9}
\centering
\includegraphics[width=\linewidth]{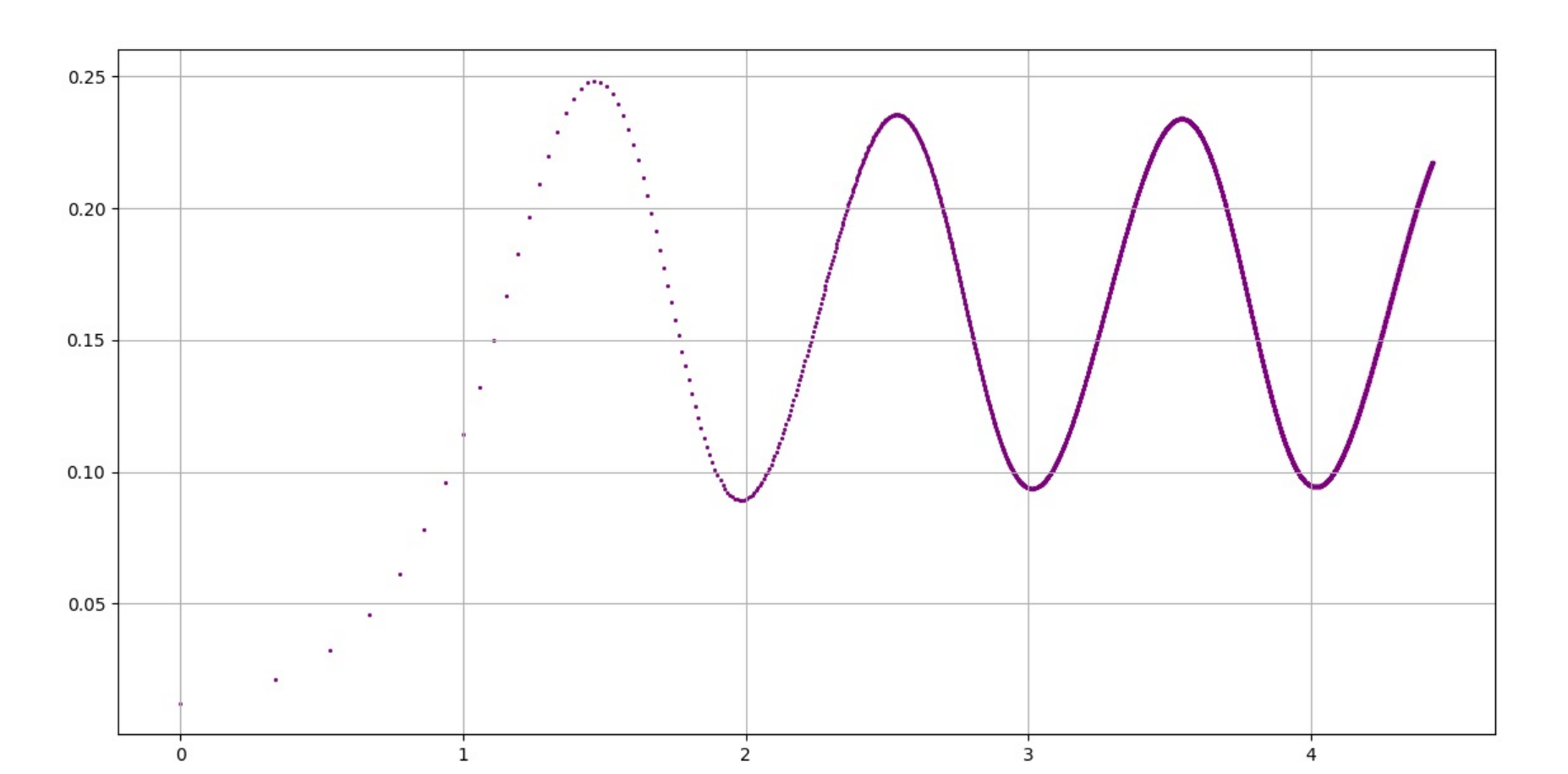}
  \end{figure}

\bigskip
\noindent
  Université de Lille,
  Faculté des Sciences et technologies,
  Département de mathématiques,
  Cité Scientifique,
  F-59655 Villeneuve d'Ascq cedex,
  France.
\newline
\strut \texttt{jean-francois.burnol@univ-lille.fr}

\end{document}